\documentclass{amsart}

\usepackage{color}
\usepackage{amsmath}
\usepackage{amsfonts}
\usepackage{amssymb,enumerate}
\usepackage{enumitem}
\usepackage{amsthm}
\usepackage{comment}
\usepackage[all]{xy}
\usepackage[utf8]{inputenc}
\usepackage[noadjust]{cite}
\usepackage{lineno}
\usepackage{mathrsfs}

\theoremstyle{definition}
\newtheorem{lem}{Lemma}[section]
\newtheorem{cor}[lem]{Corollary}
\newtheorem{prop}[lem]{Proposition}
\newtheorem{thm}[lem]{Theorem}

\newtheorem{defn}[lem]{Definition}
\newtheorem{ex}[lem]{Example}
\newtheorem{question}[lem]{Question}

\newtheorem{notation}[lem]{Notation}

\newcommand{\n}{N}
\newcommand{\no}{N_D^R}
\newcommand{\N}{N_K^F}

\newcommand{\bbn}{\mathbb{N}}

\newcommand{\J}{\mathcal{J}}

\renewcommand{\geq}{\geqslant}
\renewcommand{\leq}{\leqslant}

\newcommand{\Id}{M_{\mathcal{I}^2}}
\newcommand{\Nid}{M_\mathcal{I}}
\newcommand{\Sp}{\text{Spec}(D)}

\newcommand{\norm}[1]{\lVert#1\rVert}

\newcommand{\Max}{\text{Max}}
\newcommand{\mx}{\text{Max}}

\numberwithin{mytheorem}{subsection}

\begin{document}

\author{Jim Coykendall}
\address{School of Mathematical and Statistical Sciences\\
	Clemson University\\
	Clemson, SC 29634}
\email[J.~Coykendall]{jcoyken@clemson.edu}

\author{Richard Erwin Hasenauer}
\address{Department of Mathematics and Statistics\\
	Eastern Kentucky University\\
	Richmond, KY 40475}
\email[R. Hasenauer]{richard.hasenauer@eku.edu}

\title{Norms, Normsets, and Factorization}
\begin{abstract}
We present a development of norms and discuss their relationship to factorization.  In \cite{C96} the first named author introduced the notion of a normset, which is the image of a norm.  A normset is a monoid with its own factorization properties. We discuss in several different environments the relationship between factoring in domains and their respective normsets.  We will also discuss the utility of this notion when it comes to multiplicative ideal theory.    
\end{abstract}

\maketitle
\sloppy

\section{Introduction}
Factorization has been studied since antiquity, and in many ways can be considered one of the most central aspects of all of mathematics. 
Indeed, many aspects of modern mathematics are interwoven historically or currently with one or another aspect of factorization.  With this in mind, the aim of this paper is to discuss the notion(s) of a norm, its image monoid, and its utility in the study of factorization.

In much of modern factorization theory, the arena of interest is an integral domain or a monoid. Of course the monoid setting is perhaps most natural in the sense that it is in this realm in which factorization is stripped down into its bare essentials. Indeed, the monoid setting is where one only need worry about the binary operation that governs the factorization; one is allowed to ignore the pesky addition operation that does not seem to be germane to the issue of factorization.

Of course, domains cannot and should not be completely ignored. Indeed, it is fair to say that domains were, at least at the beginning, at the forefront of the motivating factorization questions, and even at this date there are many tantalizing questions remaining unanswered; even the question on the infinitude of real quadratic UFDs is still open (and has been since the time of Gauss).

There has been much recent progress in the theory of factorization, both in the monoid setting as well as in the domain setting, and there are striking differences in the character of the results and the techniques used. The theory of factorization in monoids is perhaps a more robust one. The setting is less restrictive and there are more sweeping statements that can be made; that pesky addition operation does inhibit structure in the case of integral domains. A dramatic illustration of this can be found in a fairly simple setting. If one considers the submonoid of the natural numbers generated by $2$ and $3$, one obtains all the nonnegative integers with the exception of $1$. In the monoid the elements $2$ and $3$ are the only irreducible elements and the monoid has the ``length-factorial" property. We briefly recall this notion.

\begin{defn}
Let $M$ be an atomic monoid. We say that $M$ is length-factorial if given any two irreducible factorization in $M$ of the same length

\[
\alpha_1\alpha_2\cdots\alpha_n=\beta_1\beta_2\cdots\beta_n
\]

\noindent then there is a permutation $\sigma\in S_n$ such that for all $1\leq i\leq n$, $\alpha_i$ is associated to $\beta_{\sigma(i)}$.
    \end{defn}

In layman's terms, this means that any element of $M$ has at most one factorization of a given length. The length-factorial property was introduced in the paper \cite{CS11} with the terminology ``other half factorial" and then rechristened and exhaustively studied in \cite{CCGS21}.

Nontrivial examples of the length-factorial property exist in abundance in the setting of atomic monoids; a characterization of length-factorial Krull monoids can be found in \cite{GZ21}; and more in \cite{BVZ23}). Concretely, the monoid mentioned above, and more generally, any nonprincipal $2-$generated submonoid of $\mathbb{N}_0$ is length-factorial, but not a UFM. This, however, cannot happen in the setting of atomic integral domains. We recall the following theorem from \cite{CS11} rewritten utilizing current terminology.

\begin{thm}
If $R$ is an atomic length-factorial domain, then $R$ is a UFD.
\end{thm}

The previous theorem highlights the fact that there can be a bit of a disconnect with what is possible in the setting of monoids and what is possible in the narrower setting of integral domains. Another significant connection is given in the following theorem from \cite{CM11} (note, in the original paper, the word ``atomic'' seems to be missing from the statement of the main theorem, although its use is explicit in the proof).

\begin{thm}
Let $M$ be any atomic, reduced, cancellative, torsion-free monoid. Then there exists an integral domain
with atomic factorization structure isomorphic to $M$.
\end{thm}

Although we have not formally defined ``factorization structure" here, the idea that should be conveyed is that any reduced, cancellative, torsion-free monoid can be realized as the ``atomic part" of some integral domain. So if $M$ is such a monoid, there exists an integral domain $R$ such that the submonoid of $R$ generated by the atoms, under unit equivalence, is isomorphic to $M$. For many such domains (in particular if $M$ is length-factorial and not a UFM), $R$ {\it must} be non-atomic. It has also been shown (see Examples 4.6 and 4.7 from \cite{HKHK04}) that there exist monoids that cannot be realized as the (reduced) monoid of nonzero elements of any integral domain.

With these facts in mind, we highlight what the authors believe to be one of the most important open questions in factorization theory.

\begin{question}
    Which monoids can be realized as the reduced monoid of an atomic integral domain?
\end{question}

In a certain sense this would be a sort of ``theory of everything" for the interplay of factorization in monoids and domains, and we suspect it to be a question that is difficult to answer in general. For now we will satisfy ourselves with the connections between domains and monoids afforded by the various flavors of norm maps.

The classical notion of the norm is the so-called Dedekind-Hasse norm from algebraic number theory and is the familiar ``product of conjugates." In the case of quadratic fields, this version of the norm is especially useful in discerning elementary factorization properties of the ring, and the norm form is crucial in many results concerning the representation of integers by quadratic forms.

The utility of the norm lies in its reduction to a (seemingly) simpler case. In most applications of the classical Dedekind-Hasse norm, the norm function maps the domain, a potentially complicated ring of integers, into what can be construed as a simpler structure, a subset of $\mathbb{Z}$. Generally, what the norm accomplishes is a transfer of factorization properties from a domain to a multiplicative monoid in which the factorization properties can be isolated and studied.

Presently, we will discuss factorization from a classical algebraic number theory. After examining the classical norm, we expand the definition beyond the classical and explore a norm that uses valuations and places to study factorization in almost Dedekind and  Pr\"{u}fer domains.

\section{Factorization in rings in integers}

The arena of rings of algebraic integers is the classical playground for factorization. It is in this theater where questions of factorization famously became an obstruction to  Lam\'{e}'s 1847 attempted proof of Fermat's Last Theorem. Indeed, if cyclotomic rings of integers of the form $\mathbb{Z}[\zeta_p]$ where $p$ is a prime and $\zeta_p$ is a primitive root of unity were always UFDs, then Fermat's Last Theorem would have been resolved almost two centuries ago. Unfortunately (or perhaps fortunately) this is not true; although this works for all primes $19$ and smaller, it fails for $23$ and {\it all} higher primes (\cite{W1997}). This failure, in turn, led to the study of ideal numbers, Dedekind domains, and generally opened up new horizons in number theory and algebra.

Even in the time of Gauss, there were observations relating structure in rings of integers to behavior of the norms of elements in the ring (which we will define more rigorously later). Famously, Gauss noted that the nonzero elements of $\mathbb{Z}$ which can be expressed as the sum of two squares formed a unique factorization monoid (using the parlance of our times), and perhaps this is one of the first connections between a ring of integers and its set of norms.

Structure should intuitively be lost in the set of norms (and to be sure, some is), but there are significant chunks of information that survive, and in this section, we will highlight valuable insights concerning (orders in) rings of integers that can be gleaned by considering the classical Dedekind-Hasse norm.

The Dedekind-Hasse norm is the familiar product of conjugates norm defined on algebraic number fields; we recall this notion here.

\begin{defn}
Let $\mathbb{Q}\subseteq K\subseteq F$ be fields with $[F:\mathbb{Q}]=n<\infty$. We define the norm function $\N:F\longrightarrow K$ via

\[
\N(\alpha)=\prod_{\sigma\in\Lambda}\sigma(\alpha)
\]
\noindent where the set $\Lambda$ denotes the set of distinct $K-$embeddings $F\longrightarrow\mathbb{C}$ and $\mathbb{C}$ denotes the complex numbers.

\end{defn}

This classical definition gives a function from $F$ down to the smaller field $K$, and it is well known that if the rings of algebraic integers of $F$ and $K$ respectively are $R$ and $D$, then $N_K^F(R)\subseteq D$. We also remark that if the extension $F/K$ is Galois then we have $\Lambda=\text{Gal}(F/K)$ consists of the automorphisms of $F$ that fix $K$.

Beyond this definition of the norm is one with a more ideal-theoretic flavor that can be defined in a more general sense. Indeed, suppose that $D$ is a Dedekind domain with quotient field $K$ and that $F$ is a finite extension of $K$. It is known (\cite{k}) that the integral closure of $D$ in $F$ is Dedekind as well, and this allows us to construct an ideal-theoretic analog of the classical Dedekind-Hasse norm. This norm we will define on ideals by first defining the norm of a prime ideal $\mathfrak{P}\subset R$. We introduce the norm presently, but first we recall a standard result, the proof of which may be found in \cite{G} or \cite{Nark2004} (among many others).

\begin{thm}
Let $D$ be a Dedekind domain with quotient field $K$ and $F$ a finite extension of $K$ with $[F:K]=n$. If $R$ is the integral closure of $D$ in $F$, then $R$ is Dedekind. What is more, if  $\wp\subset  D$ is a prime ideal, then the ideal $\wp R$ factors as

\[
\wp R=\mathfrak{P}_1^{e_1}\mathfrak{P}_2^{e_2}\cdots\mathfrak{P}_k^{e_k}
\]
\noindent with each $\mathfrak{P}_i\subset R$ a prime ideal lying over $\wp$.

\noindent Additionally, we have that 

\[
n=\sum_{i=1}^k e_if_i
\]
\noindent where $f_i:=[R/\mathfrak{P}_i:D/\wp]$, and what is more, if the extension $F/K$ is Galois, then each $e_i=e$ and each $f_i=f$ and $n=efk$.
\end{thm}

The nonnegative integers $e_i$ are called the ramification indices and the positive integers $f_i$ are called the (inertial) degrees.

This prime ideal decomposition allows us to construct the norm of a prime ideal in the Dedekind domain $R$ via:

\[
N(\mathfrak{P})=\wp^f
\]

\noindent where $\wp=\mathfrak{P}\bigcap D$ and $f=[R/\mathfrak{P}:D/\wp]$.

As $R$ is a Dedekind domain, every nonzero proper ideal of $R$ is a product of prime ideals and this allows us to define the norm globally on the ideals of $R$. With the definition of the norm on prime ideals given above, we now define the norm for any nonzero proper ideal $I\subset R$.

\begin{defn}
Suppose that $I$ is a nonzero proper ideal of $R$ with (unique) prime ideal factorization given by

\[
I=\mathfrak{P}_1\mathfrak{P}_2\cdots\mathfrak{P}_k
\]
\noindent with the prime ideals $\mathfrak{P}_i$ not necessarily distinct, then

\[
N(I)=\wp_1^{f_1}\wp_2^{f_2}\cdots\wp_k^{f_k}
\]

\noindent with $\wp_i=\mathfrak{P}_i\bigcap D$ and $f_i=[R/\mathfrak{P}_i:D/\wp_i]$.
    
\end{defn}

For rings of integers, these notions of norms are the same in that they agree on principal ideals (and so up to units in an elemental sense). In particular if $D\subseteq R$ are rings of algebraic integers, it can be shown that if $x\in R$  then $N((x))=(N_K^F(x))$.

We remark that yet another characterization of the norm in the special case in which the base ring is $D=\mathbb{Z}$ is given by 

\[
N(I)=(\vert R/I\vert)
\]
\noindent which agrees with the above definition in this case (see \cite{IR1982}).

\section{Normsets in Rings of integers}

In this section, we suppose that $K\subseteq F$ is a finite extension of algebraic number fields with corresponding rings of integers $D\subset R$. We begin here with a few properties of the norm that will prove useful in the sequel. 

Before moving on, we note at this juncture that there are transfer homomorphisms from normsets to monoids of weighted zero-sum sequences (see \cite[Theorem 7.1]{BMOS22} and \cite[Theorem 3.2]{GHKZ22} and \cite{HK14} for an earlier contribution in this direction). Transfer homomorphisms preserve, among other things, sets of lengths, elasticities, and the half-factorial property. Although the classical norm that we describe is not a transfer homomorphism in general, the interested reader might be interesting in consulting these sources, along with \cite{G22}, for further information on this and some newer results concerning elasticities as well.

The next theorem is standard and a proof can be found in \cite{hungerford}.

\begin{thm}\label{field}
Suppose that $K\subseteq E\subseteq F$ are number fields. The norm map $N_K^F: F\longrightarrow K$ enjoys the following properties
\begin{enumerate}
\item $N_K^F=N_K^E\circ N_E^F$
\item $N_K^F(\alpha)=0$ if and only if $\alpha=0$.
\item $N_K^F(\alpha\beta)=N_K^F(\alpha)N_K^F(\beta)$ for all $\alpha,\beta\in F$.
\item If $f(t)=a_0+a_1t+\cdots+t^n\in K[t]$ is the monic minimal polynomial of $\alpha\in F$ then $\N(\alpha)=((-1)^na_0)^{[F:K(\alpha)]}$.
\end{enumerate}
\end{thm}

We now focus on the restriction of the field norm to its ring of integers. We change the notation here to emphasize that the domain and codomains are restricted to the rings of integers.

\begin{thm}\label{ring}
The norm maps enjoy the following additional properties.
\begin{enumerate}
    \item $N_K^F(R)\subseteq D$.
    \item For all $x\in R$, $\no(x)\in U(D)$ if and only if $x\in U(R)$.
    \item If $\alpha\in R$ and $\no(\alpha)$ is an irreducible element of $D$, then $\alpha$ is an irreducible element of $R$.
\end{enumerate}
    \end{thm}

    \begin{proof}

    For the first statement, since any element of $\N(R)$ is both in $K$ and integral over $D$, we have that $\N(R)\subseteq D$ as $D$ is integrally closed.
        For the second statement, first suppose that $x\in U(R)$. Since there is an element $y\in R$ with $xy=1$, we merely apply the norm and obtain $1=\no(1)=\no(xy)=\no(x)\no(y)$ and so $\no(x)\in U(D)$. 

        On the other hand, suppose that $\no(x)\in U(D)$ and its minimal monic polynomial over $D$ is given by  
        \[
        d_0+d_1t+\cdots+t^n.
        \]
        \noindent By statement (4) of Theorem \ref{field}, we obtain that $d_0\in U(D)$ and as we have

        \[
d_0=-x(d_1+d_2x+\cdots+x^{n-1}),
        \]
\noindent we see that $x\in U(R)$.

Finally suppose that $\alpha\in R$ factors as $\alpha=xy$ with $x,y\in R$. Since $\no(\alpha)$ is irreducible in $D$ and $\no(\alpha)=\no(x)\no(y)$, we have that one of $\no(x)$ or $\no(y)$ is a unit in $D$. Hence by the previous statement, we have that either $x$ or $y$ is a unit in $R$ and so $\alpha$ is an irreducible element of $R$.
\end{proof}

There are a couple of key observations that can be made from the previous theorem and its proof. The first is $\no:R\longrightarrow D$ is really just a monoid homomorphism when the domain and codomain are restricted to nonzero elements. The second observation is that for the third statement of Theorem \ref{ring},  it is not really crucial to determine if $\no(\alpha)$ decomposes in $D$, it is only important to determine if it decomposes in the image of $\no$.

We now introduce the notion of the  normset (monoid). Additionally, as we have carefully introduced the notion of the norm map, we will suppress the scripts in future uses of the norm function.

\begin{notation}
The norm maps $\N$ and $\no$ will be abbreviated and written collectively as $\n$ whenever context justifies. It will also be understood that the domain of $\n$ will be the nonzero elements of $R$.
\end{notation}

\begin{defn}
Let $D\subseteq R$ be rings of algebraic integers and $\n:R\longrightarrow D$ the norm map. We define the normset $\mathfrak{N}$ of $R$ (with respect to $D$) to be $\n(R\setminus\{0\})$, the image of the nonzero elements of $R$ in $D$.
    \end{defn}

    We remark here that $\mathfrak{N}$ is a cancellative, atomic monoid (indeed, $\mathfrak{N}$ is generated multiplicatively by the norms of the irreducible elements of $R$). One of the most useful classes of normsets arises in the case in which $D=\mathbb{Z}$. For this reason, as well as for ease and brevity of exposition, we will restrict to this case from this point forward. That is, our default assumption will be that $R$ is the ring of integers of $F$, a finite extension of $\mathbb{Q}$, and $\mathfrak{N}$ will be the normset of $R$ relative to $\mathbb{Z}$, unless indicated otherwise.

    \begin{ex}\label{ne1}
Let $R=\mathbb{Z}[i]$. For this extension $\mathfrak{N}$ is precisely the set of nonzero integers which can be expressed as the sum of two squares. This multiplicative monoid has been extensively studied from a number theoretic point of view. Using the convention that $p_i$ and $q_i$ denote primes such that $p_i\equiv 1\text{mod}(4)$ and $q_i\equiv 3\text{mod}(4)$, we point out that it is well-known that $\mathfrak{N}$ is given by

\[
\mathfrak{N}=\{2^ap_1^{b_1}p_2^{b_2}\cdots p_k^{b_k}q_1^{2c_1}q_2^{2c_2}\cdots p_t^{2c_t}\vert a, b_i, c_j\in\mathbb{N}_0\}.
\]

Note that the irreducible elements of this multiplicative monoid are precisely the set of elements $\{2\}\bigcup\{p\vert p\text{ is prime 
 and } p\equiv 1\text{mod}(4)\}\bigcup\{q^2\vert q\text{ is prime 
 and } q\equiv 3\text{mod}(4)\}.$ It is now easy to observe that every element of this monoid factors uniquely into irreducibles; in other words, $\mathfrak{N}$ is a UFM (unique factorization monoid).
    
    \end{ex}

\begin{ex}\label{ne2}
We now consider the domain $\mathbb{Z}[\sqrt{-10}]$ and its normset $\mathfrak{N}$ which consists of the nonzero elements of $\{x^2+10y^2\vert x,y\in\mathbb{Z}\}$. Note that in this monoid the elements $10$, $4$, and $25$ can all easily seen to be irreducible, and this gives rise to the nonunique factorization in $\mathbb{N}$

\[
(4)(25)=(10)^2
\]
\noindent and hence this monoid is not a UFM.
    
\end{ex}

We give a final example to demonstrate a more dramatic example of the loss of unique factorization.

\begin{ex}\label{ne3}
We now consider $\mathbb{Z}[\sqrt{-41}]$. For this domain, we have that $\mathfrak{N}$ is the nonzero elements of $\{x^2+41y^2\vert x,y\in \mathbb{Z}\}$. Note that in $\mathfrak{N}$ we have the integers $45$ and $9$, but $5\notin\mathfrak{N}$; this normset is not saturated in the sense that there is an element of the quotient field of $\mathbb{Z}[\sqrt{-41}]$ with integral norm, but no integral element with that norm. Of course, such a monoid cannot have unique factorization as this lack of saturation can be exploited to give the nonunique factorization

\[
(5^2)(9)^2=(45)^2.
\]

\noindent Here the elements in parentheses are irreducible elements of $\mathfrak{N}$; we observe that here the lengths of the factorizations do not even coincide as the left side has three irreducible factors in $\mathfrak{N}$ and the right has two. These phenomena and others will be explored in more detail later.

\end{ex}

We will now look at the interplay between factorization in the normset monoid and factorization in its parent ring of integers in the Galois case. This amount of speciality is not required in general, but for present purposes makes the discussion easier. We also remind the reader that our normset $\mathfrak{N}$ is taken to be in $\mathbb{Z}$. This too can be generalized, but for Theorem \ref{main} and many of its offshoots, one would want the generalized $D$ to also be a UFD.

We begin with a lemma that demonstrates why the symmetry of the Galois case is useful. More general treatments of the relationships between irreducibles in the ring versus irreducibles in the set of norms can be found in \cite{L87} and \cite{C96b}.

\begin{lem}\label{prep}
Let $R$ be a ring of algebraic integers with quotient field $F$. If $F$ is Galois over $K$ and $R$ is a UFD, then $\pi\in R$ is irreducible (prime) if and only if $\n(\pi)$ is an irreducible element of $\mathfrak{N}$.

\end{lem}

\begin{proof}
    The fact that $\pi$ is irreducible if $\n(\pi)$ is irreducible is a porism of the proof of part (3) of Theorem \ref{ring}.

    So now assume that $F$ is Galois over $K$, $R$ is a UFD, and $\n(\pi)=\n(a)\n(b)$ for some $a,b\in R\setminus U(R)$. We further assume that we have arranged the factorization so that $\n(a)$ is irreducible in $\mathfrak{N}$. By the above, we now have that $a$ is irreducible in $R$, and since $R$ is a UFD, is prime. The norm gives rise to the factorization in $R$
    \[
    \prod_{\sigma\in\text{Gal}(F/K)}\sigma(\pi)=\prod_{\sigma\in\text{Gal}(F/K)}\sigma(a)\sigma(b).
    \]

\noindent Since $a$ is prime, it follows that $a$ divides $\sigma(\pi)$ for some $\sigma\in\text{Gal}(F/K)$. As $\pi$ is irreducible, then so is $\sigma(\pi)$ and as $a$ is prime, this means $a$ and $\sigma(\pi)$ are associates. Hence $\n(a)=\n(\sigma(\pi))=\n(\pi)$ and so $\n(b)$ is a unit. Hence $\n(\pi)$ is irreducible in $\mathfrak{N}$.
\end{proof}

\begin{thm}\label{main}
    Let $R$ be a ring of algebraic integers with quotient field $F$. If $F$ is Galois, then the following conditions are equivalent.
    \begin{enumerate}
        \item $R$ is a UFD.
        \item $\mathfrak{N}$ is a UFM.
    \end{enumerate}
\end{thm}

\begin{proof}
Suppose first that $R$ is UFD and suppose that $\n(x)\in \mathfrak{N}$. If we have a prime factorization of $x$ in $R$ given by 

\[
x=\pi_1\pi_2\cdots\pi_n
\]

\noindent then we have the irreducible factorization of $\n(x)$ in $\mathfrak{N}$ given by

\[
\n(x)=\n(\pi_1)\n(\pi_2)\cdots\n(\pi_n).
\]

\noindent To finish, we need to show that the above factorization is unique, and to this end, suppose that we have the factorizations in $\mathfrak{N}$

\[
\n(\pi_1)\n(\pi_2)\cdots\n(\pi_n)=\n(\xi_1)\n(\xi_2)\cdots\n(\xi_m).
\]

\noindent By Lemma \ref{prep}, we have that each $\pi_i, \xi_j$ is a prime element of $R$. So, if we consider the above as a factorization in $R$, the fact that $\pi_1$ is prime shows that $\pi_1$ is associated to a conjugate of some $\xi_j$ (we will say without loss of generality that $\pi_1$ is associated to a conjugate of $\xi_1$), and hence $\n(\pi_1)$ is associated to $\n(\xi_1)$ as an element of $\mathfrak{N}$. Dividing both sides of the above equation by $\n(\pi_1)$, we proceed by induction to obtain the result.

For the converse, we assume that $R$ is not a UFD. In this case, there is a nonprincipal prime $\mathfrak{P}$ and a prime $\mathfrak{Q}$ that is in the inverse class of $[\mathfrak{P}]$ in $\text{Cl}(R)$. We further stipulate that $\mathfrak{P}$ and $\mathfrak{Q}$ are chosen to be of degree $1$ (i.e. $\n(\mathfrak{P})=(p)$ and $\n(\mathfrak{Q})=(q)$ for primes $p,q\in\mathbb{Z}$) and that $p$ and $q$ are unramified; this can be done by the Chebotarev Density Theorem (see \cite[Corollary 7, p. 347]{Nark2004} for example). Since $F$ is Galois, this means that $p$ and $q$ are split primes (that is, a product of distinct degree $1$ prime ideals of $R$) as all prime ideals containing $p$ (resp. $q$) are conjugate. We also glean from this observation that as $\mathfrak{P}$ (resp. $\mathfrak{Q}$) is nonprincipal, so are all primes in $R$ containing $p$ (resp. $q$), and so no associate of either $p$ or $q$ is in $\mathfrak{N}$.
Now note that $\mathfrak{PQ}$ is principal and generated by $\alpha\in R$, and $\n(\alpha)$ is an associate of $pq$; also note that if $n=[F:K]$ then $p^n, q^n\in\mathfrak{N}$. Since the only divisors of $p^n$ (resp. $q^n$) in $\mathbb{Z}$ are powers of $p$ (resp. $q$), the factorization

\[
(pq)^n=(p^n)(q^n)
\]
\noindent demonstrates that $\mathfrak{N}$ is not a UFM.
\end{proof}

We remark here that the assumption Galois is used for the implication UFM$\Longrightarrow$UFD, but this hypothesis can be relaxed (perhaps) rather significantly. The condition required for the implication UFM$\Longrightarrow$UFD is the technical condition ``norm factorization field extension" (NFF) which may or may not be {\it all} algebraic rings of integers, but this is still open. See \cite{C96} for more details. 

The assumption that $F$ is Galois over $K$ is more critical in the other implication.

\begin{ex}\label{5}
Let $\alpha$ be a root of the polynomial $f(x)=x^5-x^3+1\in\mathbb{Q}[x]$. The field discriminant of $\mathbb{Q}(\alpha)$ is 3017 as is the discriminant of $\mathbb{Z}[\alpha]$ and so this is the full ring of integers. Modulo $3$, $f(x)$ factors as 

\[
x^5-x^3+1=(x^3+x^2+x-1)(x^2-x-1),
\]

\noindent and from this we can deduce that $\text{Gal}(\mathbb{Q}(\alpha)/\mathbb{Q})$ is isomorphic to $S_5$ (as the Galois group must contain an element that decomposes into a $3-$cycle times a transposition and $S_5$ is the only transitive subgroup of $S_5$ with this property). It can also be checked that $\mathbb{Z}[\alpha]$ is a UFD.

As an element of $\mathbb{Z}[\alpha]$, $3$ factors as $3=\alpha\beta$, where the degree of $\alpha$ is $3$ and the degree of $\beta$ is $2$. From this it follows that both $p^3$ and $p^2$ are irreducible elements of $\mathfrak{N}$. So, even though $\mathbb{Z}[\alpha]$ is a UFD, the factorization in the normset

\[
(3^3)^2=(3^2)^3
\]

\noindent shows that $\mathfrak{N}$ is not a UFM (and in fact, even the {\it lengths} of the factorizations differ; this will be explored further down the line).
\end{ex}

What really makes the previous example work is that $2+3=5$ and $2$ and $3$ are relatively prime; this is what allows the construction of the ``bad factorization" in the monoid of norms. It is interesting to note that only the numbers $1,2,3,4,$ and $6$ are the only natural numbers with the property that the smallest (nonzero) element of any partition divides the other elements of the partition. This is the machine that drives the following theorem from \cite{C96}.

\begin{thm}\label{small}
    Let $R$ be a ring of algebraic integers with quotient field $F$, and suppose $[F:\mathbb{Q}]=2,3,4$ or $6$. If $R$ is a UFD then $\mathfrak{N}$ is a UFM.
\end{thm}

We conclude this section with some practical observations on determining unique factorization in a normset. First recall the Minkowski bound used for determining the class structure of a ring of algebraic integers.

\begin{thm}\label{mink}
Let $R$ be a ring of algebraic integers; then in
every ideal class of $R$ there is an (integral) ideal of norm less than or equal to

\[
M=\frac{n!}{n^n}(\frac{4}{\pi})^{r}\sqrt{\vert d_F\vert}
\]
 with $n$ being the degree of the extension, $r$ one-half the number
of complex embeddings, and $d_F$ the discriminant of the field.
\end{thm}

\noindent This $M$ is called the {\it Minkowski bound} for $R$.

Now suppose that $\mathfrak{P}$ is a prime ideal of $R$ lying over the prime $(p)\subseteq\mathbb{Z}$ with degree $f_p$. If $M$ is the Minkowski bound for $R$, we define the set 

\[
P=\{p\in\mathbb{Z}\vert p\text{ is prime and }p^{f_p}\leq M\}.
\]

The following result gives a criterion on the normset for unique factorization. It could be considered a small improvement on \cite[Theorem 3.4]{C96}.

\begin{thm}
Let $R$ be the ring of integers of $F$, a Galois field extension of $\mathbb{Q}$. Then $R$ is a UFD if and only if for every $p\in P$, $\pm p^{f_p}\in \mathfrak{N}.$
\end{thm}

\section{The Norm Group}

Since in many cases, the normset monoid factorization behavior is in lockstep with the factorization of the ring of integers, it is natural to discover the extent of this connection. As in the previous section, we have that $R$ is the ring of integers of $F$, a Galois extension of $\mathbb{Q}$.

With Examples \ref{ne1}, \ref{ne2}, and \ref{ne3} in mind, we make the following preliminary and perhaps naive definition. This can be generalized to a wider arena, but we will again stick to normsets over $\mathbb{Z}$.

\begin{defn}
Let $R$ be the ring of integers of $F$, a Galois extension of $\mathbb{Q}$. We say that the normset $\mathfrak{N}$ is
\begin{enumerate}
\item saturated if given $x,y\in\mathfrak{N}$ such that $\frac{y}{x}\in \mathbb{Z}$ then $\pm\frac{y}{x}\in\mathfrak{N}$,
\item strictly saturated if given $x,y\in\mathfrak{N}$ such that $\frac{y}{x}\in \mathbb{Z}$ then $\frac{y}{x}\in\mathfrak{N}$, and
\item strongly saturated if given $\alpha,\beta\in R$ such that $\frac{\n(\beta)}{\n(\alpha)}\in\mathbb{Z}$, then there is a $\gamma\in R$ with $\n(\gamma)=\n(\alpha)$ and $\gamma\vert\beta$ in $R$.
\end{enumerate}
\end{defn}

The concept of saturation in the set of norms has a strong connection to factorization properties that we will see presently. It is interesting to note that the notions of ``strictly saturated" and ``strongly saturated" given above are shown to be equivalent in \cite{CT1997}. So, the strict saturation property of the set of norms, that is $\frac{\n(\beta)}{\n(\alpha)}$ is an element of $\mathbb{Z}$ implies that $\frac{\n(\beta)}{\n(\alpha)}$ is an element of $\mathfrak{N}$, actually guarantees that there is a divisor of $\beta$ in $R$ with norm coinciding with $\n(\alpha)$.

The first ``saturation" property (which may be visualized as saturation up to unit equivalence) is a bit different, and we shall develop some machinery along the way that will make generating an example to demonstrate this a relatively straightforward task.

 We now consider the following extension of the normset $\mathfrak{N}$, which we will call the {\it extended normset} (modulo unit equivalence $\sim$).

\[
\mathfrak{N}^{\text{ext}}=\{a\in R\vert (a)=\n(I)\text{ for some integral ideal }I\subset R\}/\sim
\]

We wish to reiterate for clarity $\mathfrak{N}^{\text{ext}}$ is a collection of equivalence classes in which we make no distinction between generators of principal ideals (associates). Notice that $\mathfrak{N}$ modulo unit equivalence is contained in $\mathfrak{N}^{\text{ext}}$. This allows us to make the following definition.

\begin{defn}
Using the notation from above, we define the norm group of $R$ to be $G:=Q(\mathfrak{N}^{\text{ext}})/Q(\mathfrak{N})$, where the notation $Q(X)$ denotes the quotient group of the monoid $X$.
\end{defn}

Here we will focus on details of $G$ that measure saturation and other factorization effects in $R$, but the interested reader is encouraged to consult \cite{C96b} for a deeper discussion.

This first result is a key observation with regard to connections between saturation behavior in the normset and structure of the class group. This theorem is not too difficult to prove and can be found in \cite{C96b}.

\begin{thm}
There is an exact sequence

\begin{center}
$\xymatrix{
1\ar@{->}[r]	& H\ar@{->}^\iota[r]	&\text{Cl}(R)\ar@{->}^\pi[r]	&G\ar@{->}[r]	&1}$
\end{center}

\noindent where $H=\{[I]\in\text{Cl}(R)\vert\text{$I$ has norm equal to the norm of a principal ideal}\}$, $\iota$ is the inclusion map, and $\pi$ is given by $\pi([I])=[\n(I)]$.

    \end{thm}

The previous theorem is the linchpin to the following.

\begin{thm}\label{big}
    Let $R$ be a ring of algebraic integers with quotient field $F$. If $F$ is Galois over $\mathbb{Q}$, then the following conditions are equivalent.
    \begin{enumerate}
        \item $\mathfrak{N}$ is saturated.
        \item $H=1$.
        \item $\text{Cl}(R)\cong G$.
        \item $\text{Gal}(F/\mathbb{Q})$ acts trivially on $\text{Cl}(R)$.
    \end{enumerate}
\end{thm}

The details of the proof of this result are a bit excessive for present purposes, but a complete discussion is contained in \cite{C96b}. We do record a couple of interesting applications of this result, however. We remark that the next corollary is slightly corrected version of \cite[Theorem 3.6]{C96b} (in the original version, the base field should have been specified as $\mathbb{Q}$).

\begin{cor}
    Suppose that $[F:\mathbb{Q}]=n=p_1^{a_1}p_2^{a_2}\cdots p_k^{a_k}$ with each $p_i\in\mathbb{Z}$ prime and each $a_i>0$. If $\mathfrak{N}$ is saturated then the primes $p_i$ are the only primes dividing $\vert\text{Cl}(R)\vert$, and if we write
    \[
    \text{Cl}(R)\cong S_{p_1}\oplus S_{p_2}\oplus\cdots\oplus S_{p_k}
    \]
    \noindent as the direct sum of its Sylow subgroups, then each $S_{p_i}$ is of the form
    \[
    S_{p_i}\cong \mathbb{Z}_{p_i^{b_1}}\oplus\mathbb{Z}_{p_i^{b_2}}\oplus\cdots\mathbb{Z}_{p_i^{b_s}}
    \]

    \noindent with $0\leq b_j\leq a_i$ for all $1\leq j\leq t$.
\end{cor}

    \begin{proof} [Sketch]
    We derive contradictions to the class group being of the specified form. In the first case, we suppose that there is an ideal class, $[I]$ with order relatively prime to $n$. Note that $[I]^n=[I^n]$ contains an ideal with the norm corresponding to a principal norm (since in a degree $n$ extension, the normset contains the $n^{\text {th}}$ power of any integer). But as the order of $[I]$ is relatively prime to $n$, $I^n$ is not principal and so by Theorem \ref{big}, we obtain our desired contradiction.

    In the remaining case in which the prime divisors of $n$ coincide with the prime divisors of $\vert\text{Cl}(R)\vert$, but the exponent of a Sylow subgroup exceeds the bound, a similar contradiction can be derived.
        \end{proof}

Here is another corollary that lends itself to the case of quadratic rings of integers, and in this case gives a complete classification of quadratic rings of integers in which the set of norms is saturated.

\begin{cor}
Let $F$ be a quadratic field ($[F:\mathbb{Q}]=2$) with ring of integers $R$. The following are equivalent.
\begin{enumerate}
    \item $\mathfrak{N}$ is saturated.
    \item $\text{Cl}(R)$ is either trivial or $2-$elementary abelian.
\end{enumerate}

\end{cor}

\begin{proof}
To apply Theorem \ref{big}, we claim that $\text{Cl}(R)$ is trivial or $2-$elementary abelian if and only if the Galois action on the class group is trivial. 

We first note that all quadratic fields are Galois with Galois group $\mathbb{Z}_2$. If $\sigma$ is the nontrivial conjugation automorphism, and $I$ is an ideal of $R$, then $I\sigma(I)=N(I)$ which is principal. In particular, $\sigma$ induces the automorphism of $\text{Cl}(R)$ that takes every ideal class to its inverse. Now we merely note that if the class group is $2-$elementary abelian (or trivial), then the Galois action on the class group fixes every ideal class since the class group is of exponent $2$.

If, on the other hand, $\text{Cl}(R)$ is nontrivial and not $2-$elementary abelian, then there is an ideal class of order greater than $2$. Hence $\sigma$ acts nontrivially on this class and so $\mathfrak{N}$ cannot be saturated.
\end{proof}

We close out this section with a couple of examples. The first is inspired by the previous corollary; it is natural to ask if it can be extended in a meaningful way to higher degree extensions, but in \cite[Example 3.12]{C96b}, it was shown that this cannot be extended (at least not in the naive way). More intricate, number-theoretic details of the claims concerning this example can be found in \cite{Gras73} and \cite{mGras75}.

\begin{ex}
    There is a cubic Galois extension of $\mathbb{Q}$ with class group isomorphic to $\mathbb{Z}_3\oplus\mathbb{Z}_3$ in which the subgroup of the class group consisting of ideal classes fixed by the Galois action is of order $3$. Hence, despite the fact that this extension is Galois with class group $3-$elementary abelian, the normset is not strictly saturated.
\end{ex}

Our final example, puts the notions of saturation to rest. As was mentioned, the notions of strictly saturated and strongly saturated have been shown to be equivalent (\cite{CT1997}), but the notion of saturation (up to unit equivalence) is a distinct one.

\begin{ex}
It is well-known that the ring $R:=\mathbb{Z}[\sqrt{34}]$ has class group isomorphic to $\mathbb{Z}_2$ and has fundamental unit with norm $1$. In $R$, the element $3$ has norm $9$ and the element $5+\sqrt{34}$ has norm $-9$. Since there is no element in $R$ of norm $-1$, $\mathfrak{N}$ cannot be strongly (equivalently strictly) saturated. But by Theorem \ref{big}, $\mathfrak{N}$ is saturated.
    \end{ex}

    \section{Elasticity}

    In the theory of factorization in integral domains, the notion of elasticity is a another important object of study. We recall that if $x\in R$ is an atomic element (that is, an element that can be factored into irreducibles), then we can define the elasticity of the element as follows.

    \begin{defn}
        Let $x\in R$ be an atomic element of an integral domain. We define the elasticity of the element $x$ as 
        \[
        \rho(x)=\text{sup}\{\frac{n}{m}\vert\text{ where $x$ has atomic factorizations of length $n$ and $m$}\}
        \]
    \end{defn}

If $R$ is an atomic domain (that is, a domain in which every nonzero nonunit has at least one factorization into irreducible elements), then we can define the elasticity of the domain globally.

\begin{defn}
Let $R$ be an atomic domain. We define the elasticity of $R$ by
\[
\rho(R)=\text{sup}\{\rho(x)\vert x\text{ is a nonzero nonunit of $R$}\}.
\]

\end{defn}

Of course the smallest elasticity attainable for an element (or atomic domain) is $1$, and domains with elasticity $1$ have been the subject of much study (\cite{carlitz}, \cite{Z80}, \cite{AAZ}, \cite{Co1}, \cite{Co3}, and for a survey \cite{CC}).

\begin{defn}
    We say that an atomic domain $R$ is a half-factorial domain (HFD) if given the equality of irreducible factorizations
    \[
    \alpha_1\alpha_2\cdots\alpha_n=\beta_1\beta_2\cdots\beta_m
    \]
    \noindent then $n=m$.
\end{defn}

Half-factorial domains are precisely the domains with global elasticity $1$. Of course any UFD is an HFD, and the class of HFDs generalizes the class of UFDs in a very natural way. In the setting of rings of algebraic integers there is a famous and elegant ideal-theoretic characterization of HFDs due to Carlitz. This was presented as a partial answer to a question by Narkiewicz who asked for an arithmetical characterization of class numbers larger than $1$. We present this celebrated result below.

\begin{thm}[Carlitz]\label{carl}
If $R$ is a ring of algebraic integers, then $R$ is an HFD if and only if $\vert\text{Cl}(R)\vert\leq 2$. 
\end{thm}

What is more, there is slightly more contained here that is explicitly stated. Indeed, for rings of integers of class number not exceeding $2$, the domains are partitioned neatly into two classes: class number $1$ consists precisely of the UFDs and class number $2$ consists precisely of the HFDs that do not have unique factorization.

In the more general setting, one may consider elasticity to be a numerical measure of ``how far" an atomic domain is from having the half-factorial property. Although $\rho(R)$ is notoriously difficult to compute in the general arena of atomic domains, in the setting of rings of algebraic integers, elasticity is completely understood. Before presenting the theorem that answers the elasticity question for rings of integers, we recall the notion of the Davenport constant.

\begin{defn}
    Let $G$ be a finite abelian group (written additively). The Davenport constant for this group ($D(G)$) is the smallest $n\in\mathbb{N}$ such that any collection of $n$ elements of $G$ (with possible repetition) will have a subset that sums to $0\in G$.
\end{defn}

Because of its myriad applications in factorization theory and combinatorics, the Davenport constant has been the subject of intense study by a large number of people; it is also famously hard to compute for finite abelian groups that are not $p-$groups and have more than $2$ invariant factors. See \cite{chap95} or \cite{GS92} for further information on this interesting, useful, and generally frustrating invariant.

We now produce a theorem that ``tells it all" for the elasticity question in rings of algebraic integers. The heavy lifting on this theorem was done by Valenza in \cite{V90} and was completely put to rest by Narkiewicz in \cite{Nark95}.

\begin{thm}\label{V}
Let $D(G)$ denote the Davenport constant of the (finite abelian) group $G$, and let $R$ be a ring of algebraic integers. The elasticity of $R$ is given by

\[
\rho(R)=\begin{cases}
    \frac{D(\text{Cl}(R))}{2},\text{ if $R$ is not a UFD,}\\
    1,\text{ if $R$ is a UFD.}
\end{cases}
\]

\end{thm}

Comparing Theorem \ref{V} with Theorem \ref{main}, it is natural to ask if elasticity can be measured accurately by the monoid of norms. Indeed, in the Galois case, unique factorization in the monoid and in the domain are synonymous, and this inspires the question as to whether the HFDs (elasticity $1$) are correctly detected in the monoid of norms. Of course, it also inspires the question as to whether elasticity in general is mirrored in the monoid of norms in some discernible fashion. We will answer these questions in part presently.

We begin with a theorem from \cite{JNT}, which shows that factorizations tend to be ``tighter" in the normset than in the parent ring of integers.

\begin{thm}\label{eln}
    If $F$ be a Galois extension of $\mathbb{Q}$, then $\rho(R)\geq \rho(\mathfrak{N})$.
\end{thm}

Although more freedom of factorization in $R$ might seem intuitive, we note that the Galois assumption can be critical here. If we revisit Example \ref{5}, we see that the domain in question is a UFD (and hence has elasticity $1$), but in the set of norms, we have that $3^2$ and $3^3$ are irreducible, so the elasticity of the set of norms is at least $\frac{3}{2}$.

It is also useful to point out that the inequality in Theorem \ref{eln} can be strict. In \cite{JNT} it is shown that the ring of integers $\mathbb{Z}[\sqrt{-14}]$ has elasticity $2$ and yet its set of norms has the smaller elasticity $\frac{3}{2}$.

The next theorem is the essence of several results from \cite{JNT}.

\begin{thm}\label{elchar}
Suppose $F$ is Galois over $\mathbb{Q}$. The following conditions imply that $\rho(R)=\rho(\mathfrak{N})$.
\begin{enumerate}
\item If the norm of every irreducible element of $R$ is irreducible in $\mathfrak{N}$.
\item $\rho(R)<2$.
\item $\vert\text{Cl}(R)\vert<4$.
\item $\mathfrak{N}$ is saturated.
\end{enumerate}
\end{thm}

The previous theorem can be leveraged to obtain the following corollary. 

\begin{cor}
If $R$ is the ring of integers of a Galois extension of $\mathbb{Q}$ , then $R$ is an HFD if and only if $\mathfrak{N}$ is an HFM.
\end{cor}

\begin{proof} [Sketch]
The forward direction is immediate from Theorem \ref{elchar}. For the converse, one can show, using techniques similar to those of \cite{carlitz}, that if the class number of $R$ exceeds $2$ the $\rho(\mathfrak{N})$ must be at least $\frac{3}{2}$.
\end{proof}

We conclude this section with a cautionary tale on the normset that underscores the hazards of the case in which the field extension is not Galois. Following \cite{P1977} we say that two fields $K$ and $L$ are {\it arithmetically equivalent} if they possess the same Dedekind zeta function ($\zeta_K(s)=\zeta_L(s)$). Arithmetically equivalent fields have much in common; they share the same degrees, discriminants, number of both real and complex embeddings, and prime decomposition laws over $\mathbb{Q}$. Arithmetically equivalent fields also possess isomorphic unit groups and determine the same normal closure over $\mathbb{Q}$. Surprisingly, however, two arithmetically equivalent fields need not be isomorphic. This was apparently shown first by F. Gassmann in 1925 or 1926 who gave two fields of degree 180 over $\mathbb{Q}$ that were arithmetically equivalent, but not isomorphic. The current authors were not able to find and verify the standard reference (see reference [2] from \cite{P1977}) as this reference appears with different specifics from different sources (none of which match the pagination from the back issues of Mathematische Zeitschrift); additionally, it appears that the reference refers to comments made by Gassmann on the paper \cite{H1926}. Nonetheless, the relevant information from Gassmann's work and more information on this interesting topic can be found in \cite{P1977} and \cite{P1978}.

Bootstrapping the results in \cite{P1978} one can deduce the following theorem. In a nutshell, it is possible for two arithmetically equivalent fields to have distinct normsets, despite having identical prime decomposition laws for all rational primes.

\begin{thm}
Let $K=\mathbb{Q}(\sqrt[8]{-15})$ and $L=\mathbb{Q}(\sqrt[8]{-240})$. $K$ and $F$ are arithmetically
equivalent fields with rings of integers $T$ and $R$ possessing different normsets.
\end{thm}

The natural follow-up question that one may ask is: “Do all nonisomorphic
fields have different normsets?” The answer to this is no. 
A counterexample can be realized by considering the nonisomorphic but arithmetically equivalent fields $\mathbb{Q}(\sqrt[8]{-3})$ and  $\mathbb{Q}(\sqrt[8]{-48})$. It was shown in \cite{PS1979} that the rings of integers of these fields are both UFDs. As the
normsets of both rings are generated by the norms of all the primes, and
since the prime decomposition laws are the same in arithmetically equivalent
fields, these two fields must have identical normsets.

These examples show that although the set of norms is better at distinguishing non-isomorphic fields than is the prime decomposition laws, it is still fallible. At this writing the question as to which fields can be distinguished by complete knowledge of the normset is still open.

\section {An application}

In this section, we produce an application to at least partially demonstrate the utility of the norm. This result can be found in \cite {Co3}. Here we supply a complete proof of the fact that $\mathbb{Z}[\sqrt{-3}]$ is, in fact, an HFD and then show that $\mathbb{Z}[\sqrt{-3}]$ is the unique imaginary HFD order that is not the full ring of integers. 

\begin{prop}
The domain $\mathbb{Z}[\sqrt{-3}]$ is the unique imaginary quadratic HFD that is not integrally closed.
\end{prop}

\begin{proof}

We first point out that $R:=\mathbb{Z}[\sqrt{-3}]$ is, in fact, an HFD. To see this we note that it is well known that $T:=\mathbb{Z}[\omega]$, where $\omega=\frac{-1+\sqrt{-3}}{2}$ is the integral closure of $R$ and is a UFD (indeed, the Minkowski bound is strictly less than $2$ and so a direct application of Theorem \ref{mink} gives that $T$ is a UFD). 

We next observe that if $x:=a+b\omega\in T$ a simple computation shows that either $x\in R$ (in the case that $b$ is even), $\omega^2 x\in R$ (in the case that $a$ is even), or $\omega x\in R$ (in the case that $a$ and $b$ are both odd). The upshot is that every element of $R$ is a unit multiple of an element in $T$. Similarly, it can be shown that if $x\in R$ and $r=ab$ for $a,b\in T$ then there exist $\epsilon_1,\epsilon_2\in\mathbb{N}_0$ such that $a\omega^{\epsilon_1},b\omega^{\epsilon_2}\in R$ and $\omega^{\epsilon_1+\epsilon_2}\in R$. Hence any irreducible in $R$ remains irreducible in $T$. As $T$ is a UFD, $R$ must be an HFD.

For the uniqueness part, we briefly recall that the imaginary quadratic orders take on the form (\cite{hcohn}):

\[
R_n=
\begin{cases}
\mathbb{Z}[n\sqrt{d}] & \text{ if $d\equiv 2,3\text{mod}(4)$}\\
\mathbb{Z}[n(\frac{1+\sqrt{d}}{2})] & \text{ if $d\equiv 1\text{mod}(4)$}
\end{cases}
\]

\noindent for $d<0$ and $n\in\mathbb{N}$.

Here we will only do the case $d\equiv 2,3\text{mod}(4)$, as the case $d\equiv 1\text{mod}(4)$, although slightly more computationally tedious, is essentially the same sequence of ideas. In this case we write $R=\mathbb{Z}[n\sqrt{d}]$.

We consider the element $n\sqrt{d}\in R$ and claim that this element is irreducible. To see this we consider the norm 

\[
N(n\sqrt{d})=-dn^2
\]

\noindent and recall that the norm of a general element is given by 

\[
N(x+yn\sqrt{d})=x^2-dn^2y^2.
\]

It is easy to see that $x^2-dn^2y^2\geq -dn^2$ if $y\neq 0$. We conclude that if $x+yn\sqrt{d}$ is a proper divisor of $n\sqrt{d}$ then $x+yn\sqrt{d}=x\in\mathbb{Z}$, but it is clear that this is not possible. We conclude that $n\sqrt{d}$ is irreducible in $R$, and consider the factorizations

\[
(n\sqrt{d})(n\sqrt{d})=(d)(n)(n).
\]

\noindent Since the left has two irreducible factors and $n>1$, we see that for $R$ to be an HFD it must be the case that $n$ is prime and $d=-1$; we have reduced to the case of orders of prime index in $\mathbb{Z}[i]$.

So suppose that we have an HFD of index $p$ (prime) in the Gaussian integers. We now consider the element $p+pi\in \mathbb{Z}[pi]$. Note the norm is given by 

\[
N(p+pi)=2p^2
\]

\noindent and so any proper divisor of $p+pi$ must have norm $2,p,2p,$ or $p^2$. Additionally, if one divisor has norm $k$ then the other divisor must have norm $\frac{2p^2}{k}$ and hence we can assume that our divisor has norm either $2$ or $p$. It is easy to see that there is no element of $\mathbb{Z}[pi]$ of norm $2$ or $p$. This completes the proof of this case; the other is very similar.
\end{proof}

Given the result above, we could not resist the following consequence. The above result completes the entire list of all imaginary quadratic HFD orders. The Stark-Heegner theorem (this amazing theorem has an unusual history spread out over papers by Heegner (\cite{H1952}), then Baker (\cite{B1966}) and Stark (\cite{Stark1967} and \cite{Stark1969})) gives the complete list of imaginary quadratic UFDs. Stark later gave a complete list of quadratic rings of integers with class number $2$ (\cite{Stark1975}). The previous gives us the following complete classification.

\begin{thm}
Let $d<0$ be a square free integer and $\mathbb{Q}(\sqrt{d})$ a quadratic number field with ring of integers $\mathbb{Z}[\xi]$ where 

\[
\xi=\begin{cases}
    \sqrt{d},\text{ if $d\equiv 2,3\text{mod}(4),$}\\
    \frac{1+\sqrt{d}}{2},\text{ if $d\equiv 1\text{mod}(4).$}
\end{cases}
\]

The complete list of HFD orders $\mathbb{Z}[n\xi]$ in imaginary quadratic fields is given by
\begin{enumerate}
\item $n=1$ and $d=-1,-2,-3,-7,-11,-19,-43,-67,-163$ are the HFDs that are UFDs,
\item $n=1$ and $d=-5,-6,-10,-13,-15,-22,-35,-37,-51,-58,-91,-115,\\-123,-187,-235,-267,-403,-427$ are the integrally closed HFDs that are not UFDs,
\item $n=2$ and $d=-3$ is the unique HFD that is not integrally closed.
\end{enumerate}
\end{thm}

For those interested in this subject it should be pointed out that in the real quadratic case, non-integrally closed orders that are HFDs exist in abundance. The authors recommend the very interesting paper \cite{Pollack} on the subject of the infinitude of real quadratic HFD orders. It is shown here that infinitely many exist (even an infinite family inside a single ring of integers).

We close by pointing out a couple of other applications of norms. The first is a theorem from \cite{CO2} and has been used in number of applications in the study of HFDs; the proof given therein relies on the norm. This result also follows from a more recent characterization of HFD orders that can be found in \cite{R25}.

\begin{thm}
    Let $R$ be an HFD order in a ring of algebraic integers, then the integral closure of $R$ is also an HFD.
\end{thm}

For a final application, we give a short proof of the following theorem highlighting the value of a norm-theoretic approach.

\begin{thm}
    Let $F/\mathbb{Q}$ be Galois with ring of integers $R$. If $[F:\mathbb{Q}]$ is odd and $R$ is an HFD then $R$ is a UFD.
\end{thm}

\begin{proof}
    We utilize the results of Theorem \ref{big}. If $R$ is an HFD that is not a UFD, then by Theorem \ref{carl} the class number of $R$ is precisely $2$. Hence there is a prime ideal $\mathfrak{P}$ such that $\mathfrak{P}$ is not principal, but its square is principal. As $[F:\mathbb{Q}]=\vert\text{Gal}(F/\mathbb{Q})\vert$ is odd, then $\N(\mathfrak{P})$ is an odd product of conjugates of $\mathfrak{P}$ and hence not principal. This contradiction establishes the theorem.
\end{proof}

\section{Factorization in almost Dedekind domains}

To begin, we recall a domain $D$ is said to be Dedekind if $D$ is Noetherian and for all $M\in \Max(D)$ the localization $D_M$ is a Noetherian valuation domain.  Dropping the Noetherian assumption, we get a class of domains that are called almost Dedekind.  Dropping the requirement that the domain is Noetherian and the localization is Noetherian we get the class of Pr\"ufer domains.  Thus all Dedekind domains are almost Dedekind and all almost Dedekind domains are Pr\"ufer domains.  Furthermore, an almost Dedekind domain is Dedekind if and only if it is Noetherian, and a Pr\"ufer domain is almost Dedekind if and only if it contains no idempotent maximal ideals (recall that an ideal $I$ is said to be idempotent if $I^2=I$).

Factorization in Dedekind domains has been addressed, at least indirectly, in the previous sections.  It is to this end that we will assume that (unless stated otherwise) the almost Dedekind domains we discuss in the section are not Dedekind (that is, they are not Noetherian). Constructing almost Dedekind domains is an interesting topic in its own right and has been studied fairly intensively; a very good source for this is \cite{L2006}.  Another good reference is \cite{LL2003} in which ideal factorization in almost Dedekind domains was the topic.  The material from the previous sections of this paper inspires one to ponder how to utilize norm-like techniques in multiplicative ideal theory to tackle some basic factorization questions in almost Dedekind domains.

When considering factorization in almost Dedekind domains, a central initial question would be ``when is an almost Dedekind domain atomic?" The existence of an atomic almost Dedekind domain is known and one was constructed in \cite{Gr}, but what, if anything, can we say about stronger factorization properties? 

We briefly recall that a domain is
\begin{enumerate}
\item a {\it finite factorization domain} (FFD) if it is atomic and every nonzero nonunit is divisible by only finitely many atoms (up to associates),
\item a {\it bounded factorization domain} (BFD) if it is atomic and every nonzero nonunit has a bound on the lengths of its irreducible factorizations,
\item and has the ascending chain condition on principal ideals (ACCP) if every ascending chain of principal ideals stabilizes.
\end{enumerate}

\noindent For the above properties we have the implications 
\[
\text{FFD}\Longrightarrow\text{BFD}\Longrightarrow\text{ACCP}\Longrightarrow\text{atomic} 
\]
and none of the arrows can be reversed in general (see \cite{AAZ}).

Within the class of Dedekind domains, however, it is well known that the properties FFD, BFD, ACCP, and atomic are equivalent (and in fact every Dedekind domain has all of these properties).  It is natural to ask if they are all equivalent in the class of almost Dedekind domains.

In order to address these questions, we develop a tool analogous to the Dedekind-Hasse norm of the previous sections, but instead of focusing on the automorphisms or embeddings, we consider valuations.

Let $D$ be an almost Dedekind domain.  For each maximal ideal $M \in\Max(D)$ we know $D_M$ is a Noetherian valuation domain, and we have the valuation map $\nu_M: D_M\rightarrow \mathbb{N}_0$.  We take the following definition from \cite{rH2016}.

\begin{defn}
For nonzero $b\in D$ we define the norm of $b$ to be the net
$$N(b)=\big(\nu_M(b)\big)_{M\in \Max(D)} \subseteq \prod_{M\in \Max(D)} \mathbb{N}_0.$$ 
\end{defn}

Now since each map $\nu_M$ is a valuation, we have for nonzero elements $a, b \in D$ the property $\nu_M(ab)=\nu_M(a)+\nu_M(b)$, and this leads directly to the following.

\begin{thm}
    Let $D$ be an almost Dedekind domain.  For nonzero elements $a, b \in D$ we have $N(ab)=N(a)+N(b)$ where the addition of nets is defined to be componentwise.
\end{thm}

The Dedekind-Hasse norm discussed earlier is multiplicative, and its corresponding normset is a multiplicative monoid, and we have a very similar situation for this new norm.  The norm on an almost Dedekind domain takes the multiplicative structure of the domain and transfers itself via monoid homomorphism to an additive monoid.  The identity element in the additive monoid is the zero net, and as the units of $D$ are characterized by not being contained in any maximal ideal of $D$, we see that the zero net is the image of any unit in $D$.  If $D^*$ denotes the nonzero elements of $D$, we let Norm$(D)=\{N(b) : b\in D^*\}$ and we refer to this as the normset of $D$.

We now introduce a partial ordering on Norm$(D)$ which reflects the divisibility relations in $D$.

\begin{defn}
    Let $a, b \in D^*$ we say $N(a)\leq N(b)$ if for all maximal ideals $M\in \Max(D)$ we have $\nu_M(a) \leq \nu_M (b)$.  Moreover, we say  $N(a) < N(b)$ if $N(a)\leq N(b)$ and there exists a maximal ideal $M\in \Max(D)$ with $\nu_M(a) < \nu_M (b)$.  
\end{defn}

\begin{thm}
    Let $D$ be an almost Dedekind domain with elements $a, b \in D^*$.  We have $a$ divides $b$ if and only if $N(a)\leq N(b)$.  Furthermore, $a$ is a proper divisor of $b$ (in the sense that $a$ and $b$ not associates) precisely when $N(a)<N(b)$.
\end{thm}

As opposed to the Dedekind-Hasse norm, this norm preserves all factorization properties of the domain.  We make this precise in the next theorem.

\begin{thm}
    If $D$ is an (almost) Dedekind domain and $D^{\bullet}$ the set of units of D,  then $D^* / D^\bullet \cong \text{Norm}(D).$ So if $X$ is any factorization property of a domain (e.g. UFD, HFD, FFD, BFD, or atomic) then $D$ satisfies this property if and only if the monoid $\text{Norm}(D)$ satisfies $X$.
\end{thm}

\begin{proof}
    The function $\phi:D^*/D^{\bullet}\longrightarrow \text{Norm}(D)$ given by $\phi(xD^{\bullet})=\n(x)$ is well-defined and clearly onto. Now note if $\phi(xD^{\bullet})=\phi(yD^{\bullet})$ then $\n(x)=\n(y)$ and so $x$ and $y$ are associates in $D$. Hence $xD^{\bullet}=yD^{\bullet}$ and so $\phi$ is one to one. The last statement is now obvious.
\end{proof}

Intuitively, the monoid $\text{Norm}(D)$ corresponds to the reduced multiplicative monoid of the domain $D$.

Now while these structural theorems are nice results in their own right, the utility of this norm is best seen when studying almost Dedekind domains with a nonzero Jacobson radical.  We will denote the Jacobson radical of the domain $D$ by $\mathcal{J}.$

\begin{defn}
    We say an almost Dedekind domain $D$ is bounded if for all $b\in D^*$, there exists a $\eta \in \bbn$ such that $\nu_M(b)<\eta$ for all $M\in \Max(D)$.  If an almost Dedekind domain is not bounded we call it unbounded.  If for all nonunits $b\in D^*$ there is no upper bound on the valuations of $b$, we call the domain completely unbounded.
\end{defn}

Bounded almost Dedekind domains are equivalent to SP-domains (every nontrivial ideal is a product of radical ideals) which have been studied fairly extensively (see \cite{FHL13}, \cite{VY1978}, and \cite{O2005}). In \cite{rH2016} an unbounded almost Dedekind domain was constructed; the existence of a completely unbounded almost Dedekind domain has yet to be verified but conditions on its existence have been studied in \cite{spirito2023}, and a proof of the next theorem can be found therein.

\begin{thm}
    If $D$ is a completely unbounded almost Dedekind domain, then $\J=(0)$.
\end{thm}

The previous theorem shows atomicity in almost Dedekind (not Dedekind) domains fails if the domain has a non-trivial Jacobson radical, and so the aim of \cite{rH2017} was to characterize conditions that are necessary for an almost Dedekind domain to be atomic. For proofs of the theorems in the remainder of this section see \cite{rH2017}.

Let $D$ be an almost Dedekind domain and let $b\in D$; we set $Z(b)$ to be the set of nonunit divisors of $b$.  We say $Z(b)$ is finitely covered if there exist a set of maximal ideals $M=\{M_1, M_2, \cdots, M_n\}$ such that if $a\in Z(b)$ then $a\in M_i$ for some $i=1, 2, \cdots, n.$  If this condition holds for all $b\in D^*$, then we say $D$ is finitely coverable.  This leads to the following theorems. 

\begin{thm}  Let $D$ be an almost Dedekind domain.  If $D$ is finitely coverable, then $D$ is a BFD.
\end{thm}

More topological ideas were presented in \cite{rH2017} that led to the discovery of a class of almost Dedekind domains where the BFD property and the ACCP property are equivalent.   The motivating idea was that the $Z(b)$ should not contain ``too many" comaximal elements.  Precise definitions and theorems are presented below.

\begin{defn} Let $D$ be an integral domain and let $b\in D^*$.   We say $Z(b)$ is disconnected if there exists $\{a_i\}_{i=1}^{\infty} \subseteq Z(b)$ such that $\mx(a_i)\cap \mx(a_j)= \emptyset$ whenever $i\not= j$, that is to say there is an infinite collection of divisors of $b$ all of which are comaximal to one another.  We say $Z(b)$ is connected if it is not disconnected.
\end{defn}

\begin{defn}   We say an integral domain $D$ is connected if for all $b\in D$, $Z(b)$ is connected.  We will say $D$ is disconnected if there exists $b\in D$ such that $Z(b)$ is disconnected.
\end{defn}

\begin{thm}  If $D$ satisfies ACCP, then $D$ is connected.
\end{thm}

Combining the ideal of connectivity and finitely coverable we get a large class of domains that do not distinguish between ACCP and BFD.  

\begin{thm}  Let $D$ be an almost Dedekind domain with a finite set of non-invertible maximal ideals, say $S=\{M_1, M_2, \cdots, M_l\}$.  The following conditions are equivalent.

\begin{itemize}
\item[i) ] $D$ is connected
\item[ii) ] $D$ satisfies ACCP
\item[iii) ] $D$ is a BFD.
\end{itemize}
\end{thm}

The only atomic almost Dedekind domains in the literature are ACCP and BFD, and \cite{rH2017} lays out the conditions that would need to fail in order to have an atomic almost Dedekind domain that is not ACCP. It may be the case that these notions are equivalent in the class of almost Dedekind domains, and it would be nice to see a proof of their equivalence or an example showing they are not equivalent.

\section{Factorization in Pr\"{u}fer domains}

Building upon the works discussed in the previous sections, we extend our study to the class of  Pr\"{u}fer domains (often $1-$dimensional, but we will carefully differentiate as not all the results require this dimension restriction).  Recall a domain $D$ is said to be Pr\"{u}fer if $D_M$ is a valuation domain for all $M\in \Max (D)$.  Unlike in the class of almost Dedekind domains, we are not guaranteed that the associated value groups with each localization is discrete.  For example, if $M$ is an idempotent maximal ideal, we know the associated value group cannot be discrete.  And since a $1-$dimensional Pr\"{u}fer domain that is not almost Dedekind must contain at least one idempotent maximal ideal, we need to develop tools to work in this new arena.  Most of the results in this section can be found in the paper \cite{CH2018}.

To begin, we define a norm on the elements of a general Pr\"{u}fer domain in a fashion similar to the definition in the previous section.  Since Pr\"{u}fer domains need not be $1-$dimensional, we will define two norms.  For $P\in \Sp$ we let $\nu_P: D_P\setminus\{0\} \rightarrow G_P$ denote the local valuation map into the value group $G_P$ (written additively).  Also since $\Max(D) \subseteq \Sp$, $G_M=G_P$ if $P$ happens to be maximal.

\begin{defn}  Let $D$ be a Pr\"{u}fer domain.  For nonzero $b \in D$ we define
$$N(b)=(\nu_M(b))_{M\in \Max(D)} \subseteq \prod_{M\in \Max(D)} G_M$$ and
$$\bar{N}(b)=(\nu_P(b))_{P\in \Sp} \subseteq \prod_{P\in \Sp} G_P.$$
\end{defn}

From the properties of valuations we see that $N(ab)=N(a)+N(b)$ and $\bar{N}(ab)=\bar{N}(a)+\bar{N}(b)$, where addition is defined componentwise.

We once again define the associated normsets in the same way.  That is $$\text{Norm}(D)=\{N(b) : b\in D\setminus\{0\}\}$$
and
$$\overline{\text{Norm}(D)}=\{\bar{N}(b) : b\in D\setminus\{0\}\}.$$

We observe that both $\text{Norm}(D)$ and $\overline{\text{Norm}}(D)$ form additive monoids with the identity element being the zero net.  

An important observation is that $\overline{\text{Norm}(D)}$ contains ``more", in the sense of nonmaximal primes, components (unless $D$ is almost Dedekind).  But these extra components do not relay more information. This interesting observation is summarized in the following theorems.

\begin{thm} $D^{*}/U(D) \cong \text{Norm}(D)$ and $D^{*}/U(D) \cong \overline{\text{Norm}}(D)$.  Hence $\text{Norm}(D) \cong \overline{\text{Norm}}(D)$.
\end{thm}

So essentially this theorem tells us that we really only need to consider localizations at the maximal ideals. This result mirrors the well-known result that Pr\"{u}fer domains are determined by local behavior at the maximal ideals (and so one does not need to know the behavior on all of $\text{Spec}(D)$).

We again define a partial ordering on the elements of the normset in the following way.

 \begin{defn}  We say $N(a)\leq N(b)$ if for all $M \in \Max(D)$ we have $\nu_M(a)\leq \nu_M(b)$.  We say $N(a)<N(b)$ if $N(a) \leq N(b)$ and there exists an $M\in \Max (D)$ with $\nu_M(a) < \nu_M(b).$
\end{defn}

With this ordering we make the following observation.

\begin{thm}
Let $D$ be a Pr\"{u}fer domain with $a,b\in D^*$, then $a\vert b$ if and only if $N(a)\leq N(b)$. Furthermore $a$ is a proper divisor of $b$ if and only if $N(a)<N(b)$
\end{thm}

This simple proposition powers the proof of the following theorem, which at its essence tells us we can really study just the normsets to understand factorization in the domain. The following is \cite[Theorem 2.5]{CH2018}.

\begin{thm}  Let $\mathscr{F}$ be any one of the conditions a) unique factorization, b) half-factorial, c) bounded factorization, d) finite factorization, e) ACCP, f) atomic.  A Pr\"{u}fer domain $D$ has factorization property $\mathscr{F}$ if and only if $\text{Norm}(D)$ has factorization property  $\mathscr{F}$.
\end{thm}

For $1-$dimensional Pr\"{u}fer domains, we make a crucial definition. The $1-$dimensionality is needed for comparability in $\mathbb{R}$.

\begin{defn}  Let $D$ be a $1-$dimensional Pr\"{u}fer domain.   We say $D$ is uniformly bounded if for all nonzero nonunit $b\in D$, there exists $\delta>0, \eta>0 \in \mathbb{R}$ such that for all $M \in \Max(D)$ we have  $\delta<\nu_M(b)<\eta$.
\end{defn}

With this in hand, we now get a result similar to our atomicity result for almost Dedekind domains in this more general case.

\begin{thm}  If $D$ is an atomic $1-$dimensional Pr\"{u}fer domain with $\J\neq(0)$, then $D$ is a semilocal PID.
\end{thm}

\begin{proof}
    Let D be atomic and let $\J\neq(0)$. It remains to show that D is a semilocal
PID. It follows from \cite[Corollary 6.2]{O17} that D is a B\'ezout domain, and hence D is
a GCD-domain. Since D is an atomic GCD-domain, we have that D is a UFD, and
thus D is a PID (since D is 1-dimensional). Since D is a PID and $\J\neq(0)$, it follows
that D is semilocal (since D is of finite character). 
\end{proof}

Our reduction to the $1-$dimensional case was implemented to take advantage of the Archimedean property. As $1-$dimensional Pr\"{u}fer domains are Archimedean, their associated value groups can be realized as an additive subgroup of the real numbers.  This leads to the notion of another norm-like object.  In order for this new norm-like object to have meaning on all the elements of the Pr\"{u}fer domain we need to restrict the class of domains in which we are going to work.  We say a domain is of finite character if every element is contained in only finitely many maximal ideals; we restrict our study to $1-$dimensional Pr\"ufer domains of finite character. 

\begin{defn}  Let $D$ be a $1-$dimensional Pr\"ufer domain of finite character.  For $b\in D$ we define the length of $b$ to be $\norm{b}=\sum_{M\in \Max(D)} \nu_M(b)$.
\end{defn}

It is clear that $\norm{ab}=\norm{a} + \norm{b}$.  Further, if $a\vert b$ then $\nu_M(a)\leq \nu_M(b)$ for all $M \in \Max(D)$ thus $\norm{a}\leq \norm{b}$.  The converse is clearly not true, as we can have $\norm{a}\leq \norm{b}$ with $a$ and $b$ being comaximal.

We now produce a theorem that shows that the BFD property and the ACCP property cannot be distinguished in this class of domains.  Moreover, it shows that in this particular class of Pr\"ufer domains, atomic can be added to the list of equivalent conditions.  The proof relies heavily on the fact that the domain is of finite character, meaning the norm of each element is 0 in all but finitely many components.  We need one more definition before we can state the theorem; recall that $Z(b)$ is the set of nonunit divisors of $b$.

\begin{defn}  Let $D$ be a $1-$dimensional Pr\"ufer domain of finite character.  For $b\in D$ we define the set of all lengths of all divisors to be $S_b=\{\norm{d} : d \in Z(b) \}.$
\end{defn}

\begin{thm}
Let $D$ be a $1-$dimensional Pr\"ufer domain of finite character.  The following are equivalent:
\begin{itemize}
\item[(i) ] For all $b\in D$, $\inf S_b>0$.
\item[(ii) ] $D$ is a BFD.
\item[(iii) ] $D$ satisfies ACCP.
\item[(iv) ] $D$ is atomic.
\end{itemize}
\end{thm}

\begin{proof}  Let $b\in D$ and assume $\inf S_b=t>0$.  Now for all $d\in Z(b)$ we have $\norm{d}\geq t$.  Thus $\pi(b)=\lceil\frac{\norm{(b)}}{t}\rceil$ is a bound on the length of all possible factorizations of $b$. Thus $D$ is a BFD.

Clearly $(\text{ii})$ implies $(\text{iii})$ and $(\text{iii})$ implies $(\text{iv})$.

Now we assume that $D$ is atomic and let $b\in D$ with $|\mx(b)|=k$.  Now set $H=\{A \in \mathcal{P}(\mx(b))
: \text{ there exists an atom } a \text{ with } \mx(a)=A\}$, where $\mathcal{P}(\mx(b))$ is the power set of $\mx(b)$.  Since $H$ is finite we will rewrite $H=\{A_1, A_2, \cdots A_l\}$.  Now we find atoms $a_1, a_2, \cdots, a_l$ with $\mx(a_i)=A_i$.

Now set $t_i=\min\{\nu_M(a_i): M \in \mx(a_i)\}$ and set $T=\min\{t_1, t_2, \cdots, t_l\}$.  Now we claim that $\xi=\frac{T}{2}$ is a lower bound for $S_b$.

Suppose $d\in Z(b)$ with $\norm{d}=\xi$.  Now since $D$ is atomic, $d$ must be divisible by some atom $a$.   It must be the case that $\mx(a)=A_i$ for some $i$.  We have $\xi<\nu_M(a_i)$ for all $M\in \mx(a_i)$.  But this is impossible as it implies $d$ divides the atom $a_i$.  Therefore we conclude that $S_b$ has a lower bound.
\end{proof}

One would have no reason to believe that atomic would imply BFD (indeed, this is far from true in a general setting), but it is true in this class of Pr\"ufer domains.

One might also ask about finite factorization property.  Recall that a domain $D$ is said to be a finite factorization domain (FFD) if for all nonzero nonunits $b\in D$ we have the set of nonunit divisors of $b$, $Z(b)$, being finite.  The Pr\"ufer domain constructed in \cite{Gr} is in fact an FFD.  To verify this, we develop a characterization of all FFDs in the class of $1-$dimensional Pr\"ufer domains of finite character.

Let $D$ be a Pr\"ufer domain.  We let $\Id=\{M\in \Max(D): M=M^2\}$ and  $\Nid=\{M\in \Max(D): M \not= M^2\}.$

\begin{defn} We say $\Id$ is covered by non-idempotents if $$\cup_{M\in \Id} M \subset \cup_{M\in \Nid} M.$$
\end{defn}

\begin{thm} Let $D$ be a $1-$dimensional Pr\"ufer domain of finite character. If $\Id$ is covered by non-idempotents, then $D$ is a BFD.
\end{thm}

\begin{proof}  Take $b \in D$ then $\inf S_b=1$, since every $b$ is contained in an $M$ with $G_M\cong \mathbb{N}_{0}.$  Thus $D$ is a BFD.
\end{proof}

Now if the set of idempotent maximal ideals is a singleton set, this idea of the idempotent maximal ideals being covered by the non-idempotent maximal ideals is enough to force the FFD property.

\begin{thm}  Let $D$ be a $1-$dimensional Pr\"ufer domain of finite character with $\Id=\{M\}$.  If $M$ is  covered by non-idempotents, then $D$ is an FFD.\end{thm}

\begin{proof}

Suppose $M$ is  covered by non-idempotents.  Let $b\in D$.  Now if $\mx(b)\subset \Nid$ the number of divisors of $b$ is bounded by $\prod_{M\in \mx(b)} (\nu_M(b)+1)$.  Suppose $b\in M$ and $b$ has infinitely many divisors, then there must exist  $c, d \in Z(b)$ with $\nu_P(c)=\nu_P(d)$ for all $P \in \mx(b)\setminus \{M\}$.  Now without loss of generality we may assume that $\nu_M(c)<\nu_M(d)$.  But now $\mx(\frac{d}{c})=\{M\} $ which is a contradiction.
\end{proof}

Now whether or not this is true if $\Id$ has more than one element is unknown.  The proof presented in \cite{CH2018} will not work in these cases.

It would be nice to see if one can construct a $1-$dimensional Pr\"ufer domain of finite character that is an FFD that contains more than one idempotent maximal ideal.  It is quite easy to construct  a $1-$dimensional Pr\"ufer domain of finite character that is an BFD that contains more than one idempotent maximal ideal, by mimicking the example laid out in \cite{Gr}.

At this time, it is also unknown whether any of these factorization properties can hold in a Pr\"ufer domain of dimension greater than one.

\section{Ideal class (semi)groups}

Let $D$ be a Pr\"ufer domain with fraction field $K$.   We denote the set of fractional ideals of $D$ by $\mathcal{F}(D)$ and the set of principal ideals of $D$ by $\mathcal{P}(D)$.  The ideal class group of $D$ is the quotient group $\mathcal{C}(D)=\mathcal{I}(D)/\mathcal{P}(D)$, where $\mathcal{I}(D)$ is the set of invertible fractional ideals of $D$.  The ideal class (semi)group of $D$ is the quotient group $\mathcal{S}(D)=\mathcal{F}(D)/\mathcal{P}(D)$.  It should be noted that if $D$ is Dedekind then $\mathcal{S}(D)$ is a group.  If $D$ contains noninvertible ideals, which is the case in almost Dedekind and Pr\"ufer domains (that are not Dedekind), then $\mathcal{S}(D)$ is a semigroup.

To study ideal class (semi)groups we first need to discuss how to form a norm on the ideals of a Pr\"ufer domain.  There is more than one way to do this as we will discuss, but we will stick with how the norm was constructed in \cite{CH2018}.

 Let $M$ be a maximal ideal of a Pr\"ufer domain $D$ and let $\gamma \in G_M$, we define

$$^{\gamma}M=\{b\in D : \nu_M(b)\geq \gamma \}.$$
and
$$\overline{^{\gamma}M}=\{b\in D : \nu_M(b) > \gamma \}.$$

Both $^{\gamma}M$ and $\overline{^{\gamma}M}$ are ideals of $D$.  Moreover they can be distinct ideals of $D$, and when this happens, we need to create a mechanism to distinguish these two ideals.  In \cite{CH2018} the surreal numbers, $\mathbb{S}$, were used, but speaking loosely, there are other options available. What is really needed is a big set with a way to create an ``on/off" switch, as we will see presently.

First we need to make some definitions.  As discussed before, we will restrict ourselves to $1-$dimensional Pr\"ufer domains so we can think of the value groups as additive subgroups of the real numbers.

\begin{defn}
    Let $D$ be a $1-$dimensional Pr\"ufer domain and for $M \in \Max(D)$ and an ideal $I$, we define $$T_M(I)=\{\nu_M(b)\}_{b\in I\setminus\{0\}}$$ and we set $$s_M(I)=\inf \, T_M(I).$$
\end{defn} 

We need to distinguish when $s_M(I)$ is in $T_{M}(I)$ and when it is not (hence the idea of an ``on/off" switch).  To do this we use the surreal number $\epsilon=\frac{1}{\omega}$.  Here $\omega$ is the cardinality of the natural numbers.

\begin{defn} Let $D$ be a $1-$dimensional Pr\"ufer domain and let $I$ be a nonzero ideal of $D$.  We define the value of $I$ at the maximal ideal $M$ as
$$\nu_M(I) = \left\{
     \begin{array}{lr}
       s_M(I) & :s_M(I) \in T_M(I)\\
       s_M(I)+\epsilon & : s_M(I) \not\in T_M(I)\\

     \end{array}
   \right.$$
   \end{defn}

    We observe that if $\nu_M(I)=s_M(I) + \epsilon$ we have $\gamma>s_M(I)+\epsilon$ for all $\gamma \in T_M(I)$. For $s, t \in \mathbb{S}$ we say $s \sim_M t$ if and only if  $^t M=\, ^s M$. 
    
     We take full advantage of the fact that $m\epsilon \sim_M n \epsilon$ for all $m, n \in \mathbb{N}$, and if $G_M \cong\mathbb{N}$ then, $\epsilon \sim_M 1$.  
     
     We now assign a value to our ideal for every maximal ideal in $D$.  For $M\in \Max(D)$ and nonzero ideal $I$ of $D$, we will give a value to $I$ from $S_{G_M}=\mathbb{S}/\sim_M$.   We will now use ideals of the form $$^{\gamma}M=\{b\in D : \nu_M(b)\geq \gamma \}$$ where $\gamma \in S_{G_M}$ to construct our norm.

The following is \cite[Lemma 2.7]{CH2018}, the proof of which, although not terribly difficult, is a bit lengthy and benefits from the context of the discussion in \cite{CH2018}.

     \begin{lem}  Let $I$ and $J$ be ideals of a $1-$dimensional  Pr\"ufer domain.   Then $\nu_M(IJ)=\nu_M(I) + \nu_M (J)$ for all $M\in \Max(D)$.
   \end{lem}

   This puts us in a position to define our norm.

   \begin{defn}  Let $D$ be an $1-$dimensional Pr\"ufer domain, and let $I$ be an ideal of $D$.  The norm of $I$ is defined to be
$$\hat{N}(I)=(\nu_M(I))_{M\in \Max(D)}\subseteq \prod_{M\in Max(D)} S_{G_M}.$$
\end{defn}

It is probably worth taking a step back to realize that this norm on the class of almost Dedekind domains would never have to involve a surreal number, because the localizations are discrete; it is impossible for $s_M(I) \not \in T_M(I).$

Now while this construction is interesting in its own right.  It can be a useful tool in studying the ideal class semigroups of Pr\"ufer domains.  Given that an almost Dedekind domain that is not Dedekind must necessarily have an infinite number of maximal ideals, it is close to impossible to completely describe the ideal class semigroup.  Even more problems arise when considering at Pr\"ufer domains that are not almost Dedekind.

With that being said we are able to use the ideal norm described above to prove the following which are, respectively, Theorem 3.7 and Theorem 3.8 from \cite{rH2021}.

\begin{thm}  Let $D$ be an almost Dedekind domain that is not Dedekind with finitely many non-invertible maximal ideals.   If $D$ is atomic then, $\mathcal{C}(D)$ must be of infinite order.  \end{thm}

\begin{thm}  Let $D$ be an atomic Pr\"ufer domain of finite character, that is not Dedekind.   If $D$ is atomic, then $\mathcal{C}(D)$ is of infinite order.
   \end{thm}

   The first theorem might not be that surprising since we know that an almost Dedekind domain cannot be of finite character unless it is Dedekind.  The second is perhaps a bit more interesting, as one might expect the finite character property to control the ideal class semigroup.  Proofs of both of these theorems can be found in \cite{rH2021}.

   Of course, this behavior is quite different from what happens in the Dedekind case.  If $D$ is Dedekind then $\mathcal{C}(D)$ is a measure (in some sense) of how far away the domain is from being a UFD.  The smaller the size of $\mathcal{C}(D)$ the ``closer" $D$ is to being a UFD.  Notice that if $D$ is Dedekind, then $D$ is a UFD if and only if $\mathcal{C}(D)$ is trivial.

   To show the difficulty of computing an ideal class semigroup, we will develop and include an example from \cite{rH2021}. An almost Dedekind domain is said to be a sequence domain if it has a countable number of maximal ideals, with all maximal ideals being principal with the exception of one non-invertible maximal ideal.   The following theorem, which can be found in \cite{G} can be used to construct a wide array of almost Dedekind and Pr\"ufer domains.  Here we use it to construct a sequence domain, which is one of the simplest examples of an almost Dedekind domain.

\begin{thm}\label{gi} Let $D$ be a Dedekind domain with quotient field $K$, and let $\{P_i\}_{i=1}^{r}$, $\{Q_i\}_{i=1}^{s}$, and $\{U_i\}_{i=1}^{t}$, where $r\geq 1$, be three collections of distinct maximal ideals of $D$, each with finite residue field.  Then there exists a simple quadratic extension field $K(t)$ of $D$ with $t$ integral over $D$ and separable over $K$ such that if $\bar{D}$ is the integral closure of $D$ in $K(t)$, each $P_i$ is inertial with respect to $\bar{D}$, each $Q_i$ ramifies with respect to $\bar{D}$, and each $U_i$ decomposes with respect to $\bar{D}$.
\end{thm}

\begin{ex}\label{mess}
Let $D=\mathbb{Z}_{(q)}$ for some prime $q\in\mathbb{N}$. Let $K$ denote the quotient field of $D$. We begin the construction by finding a quadratic extension of $D$ where $(q)$ splits into two distinct (principal) primes, say $K_1$.  In $D_1$, the integral closure of $D$ in $K_1$, we have $(q)=(q_{1})(\omega_1)$ where $(q_{1})$ and $(\omega_1)$ are distinct primes. At the second stage we find $K_2$, a quadratic extension of $K_1$ in which $q_1$ is inert and $\omega_1$ splits. In $D_2$, the integral closure of $D_1$ in $K_2$, we have $(\omega_1)=(q_2)(\omega_2)$. Inductively in $D_n$ we have primes $q_1,q_2,\cdots,q_n,\omega_n$, and using Theorem \ref{gi} we construct $K_{n+1}$ a quadratic extension of $K_n$ such that in $D_{n+1}$ we have primes $q_1,q_2,\cdots,q_n,q_{n+1},\omega_{n+1}$. If we define $D^\infty:=\bigcup_{n=1}^\infty D_n$ it can be shown that $D^\infty$ is almost Dedekind, and even more, is a sequence domain. $\text{MaxSpec}(D^\infty)=\{ (q_n)\vert n\in\mathbb{N}\}\bigcup\{M\}$, where $M=(\omega_1,\omega_2,\omega_3,\ldots)$ is the unique maximal ideal that is not finitely generated. For more on sequence domains see \cite{L1997}.
\end{ex}

In order to calculate the ideal class semigroup of $D^\infty$ we use the fact that sequence domains are B\'ezout (meaning that $\mathcal{C}(D^\infty)$ is trivial).  From the construction it is not difficult to see that every element of $D^\infty$ is bounded, thus $D^\infty$ is an SP-domain, and we know that every ideal in an SP-domain factors into finite product of radical ideals (\cite{L1997}).  Thus we know the set of radical ideals will generate the ideal class semigroup of $D^\infty$.  

From Example \ref{mess}, which of these radical ideals correspond to non-invertible ideals? Let $\{a_k\}_{k=1}^{\infty}$ be a non-convergent sequence of zeros and ones; that is $a_k=0, 1$ for all $i$, and moreover the sequence never becomes constant.  Now the collection of all sequences of this form is uncountable.  For each possible sequence $\{a_k\}$, we define 
$s_i=j$, where $j$ is the index of the entry in $\{a_k\}_{k=1}^\infty$ containing the $i^{\text{th}}$ occurrence of $1$, and we further declare that $S:=\{s_i\}_{i=1}^\infty.$

Now let $$I_S=(q, \frac{q}{q_{s_1}}, \frac{q}{q_{s_1}q_{s_2}}, \frac{q}{q_{s_1}q_{s_2}q_{s_3}}, \cdots)$$ be a non-finitely generated ideal.  We note that $I$ is contained in $(q_j)$ for all $j$ such that the $j^{\text{th}}$ entry of $\{a_k\}$ is $0$ (which is an infinite set). Now for a $b\in D^\infty$ that is contained in infinitely many maximal ideals, there exists an $N\in \mathbb{N}$ such that $b\in (q_n)$ for all $n\geq N$. If the sequence becomes constant, then we obtain either an ideal of the form $I=(\omega_n)M \equiv M \pmod{\mathcal{P}(D^\infty)}$ or an ideal that is a finite product of some of the primes $(q_m)$.

 Thus, the original $I_S$ is not invertible, and moreover for every non-convergent sequence we get a non-invertible ideal. It should also be pointed out that $I_S \subset M$ for all $S$.
  
  We claim that $H=\{I_{S_\lambda}\vert \lambda \in \Lambda\}\cup{\{M\}}$ is a generating set for the non-identity elements of the ideal class semigroup, where $\Lambda$ is some uncountable index set corresponding to the non-convergent sequences of zeros and ones.   First note if $S_\lambda \triangle S_\gamma=(S_\lambda \setminus S_\gamma) \cup (S_\gamma \setminus S_\lambda)$ is finite, then $I_{S_{\lambda}}\equiv I_{S_{\gamma}} \pmod{\mathcal{P}(D)}$.  Thus we are not claiming that $H$ is a minimal generating set.
  
  \begin{thm}  Let $I\subset D^\infty$ be a non-invertible ideal, then $I\equiv I_1 I_2\cdots I_k \pmod{\mathcal{P}(D^\infty)}$ for some $I_1, I_2, \cdots, I_k \in H$.   This representation is not unique and the ideals in the product may not be distinct.
  \end{thm}
  
  This gives a description of a generating set for $\mathcal{S}(D^\infty).$  The equivalence classes can be described under the relation  $S_\lambda \equiv S_\gamma$ if and only if $S_\lambda \triangle S_\gamma$ is finite, and with $M$ included this forms our generating set; a proof of this can be found in \cite{rH2021}.

\section*{Acknowledgment}

The authors are truly grateful to the referee for a detailed and helpful report that was produced with incredibly baffling speed.

\bibliography{biblio2}{}
\bibliographystyle{plain}

\end{document}